\documentclass{amsart}
\usepackage[mathscr]{eucal}
\usepackage{amsmath,amssymb,latexsym,amscd}   
\usepackage[all,cmtip]{xy}
\usepackage{mathrsfs}
\usepackage{graphicx}
\usepackage[all,cmtip]{xy}
\usepackage{hyperref}
\hypersetup{colorlinks=true,linkcolor=blue,citecolor=red,linktocpage=true,pagebackref=true,hyperindex=true}

\DeclareMathOperator{\tors}{tors}
\DeclareMathOperator{\Mod}{Mod}

\DeclareMathOperator{\lep}{lep}

\theoremstyle{plain}
\newtheorem{thm}{Theorem}[section]
\newtheorem{cor}[thm]{Corollary}
\newtheorem{lem}[thm]{Lemma}
\newtheorem{prop}[thm]{Proposition}

\theoremstyle{definition}
\newtheorem{dfn}[thm]{Definition}
\newtheorem{obs}[thm]{Remark}
\newtheorem{ej}[thm]{Example}

\begin{document}

\title{On some operators and dimensions in modular meet-continuous lattices}

\author{Mauricio Medina B\'arcenas, Jos\'e R\'\i os Montes and Angel Zald\'\i var Corichi}

\address{Instituto de Matem\'aticas\\ 
UNIVERSIDAD NACIONAL AUTONOMA DE MEXICO \'Area de la Investigaci\'on Cient\'ifica, Circuito exterior, Ciudad Universitaria\\ 04510 M\'exico, D.F\\M\'exico}

\email{mmedina@matem.unam.mx\\ jrios@matem.unam.mx\\ zaldivar@matem.unam.mx}


\def\xypic{\hbox{\rm\Xy-pic}}

\maketitle

\begin{abstract}
Given a complete modular meet-continuous lattice $A$, an inflator on $A$ is a monotone function $d\colon A\rightarrow A$such that $a\leq d(a)$ for all $a\in A$.  If $I(A)$ is the set of all inflators on $A$, then $I(A)$ is a complete lattice. Motivated by preradical theory we introduce two operators, the totalizer and the equalizer. We obtain some properties of these operators and see how they are related to the structure of the lattice $A$ and with the concept of dimension.

\end{abstract}

\section{Introduction}
\label{sec:sec1}
In the study of rings and module categories certain kind of lattices emerge naturally. These lattices control many behaviors that appear on rings and module categories, for instance, the lattices of submodules, lattices of classes of modules (specially hereditary torsion classes) and the lattices of preradicals. Most of these lattices have an extra operation, via this operation one can define interesting elements. For example it is known that in the case of the lattice of preradicals over an associative ring with unit, there is a product and a coproduct. The product of preradicals $\tau$ and $\sigma$ is just the composition $(\sigma\tau)(M)=\sigma(\tau(M))$ and the coproduct is the preradical $(\sigma\colon\tau)$ given by the submodule $(\sigma\colon\tau)(M)$ such that $(\sigma\colon\tau)(M)/\sigma(M)=\tau(M/\sigma(M))$ with $M$ a module. The authors in \cite{17} introduce for a preradical $\sigma$ two prearadicals associated to each one of these two operation respectively, these preradicals are the \emph{equalizer} of $\sigma$, the \emph{annihilator} of $\sigma$ (these two associated to the product); the \emph{co-equalizer} of $\sigma$ and \emph{the totalizer} of $\sigma$ (these associated to the coproduct). In particular for the coproduct the following comparison holds $\sigma\vee \tau\leq (\sigma\colon\tau)$ as in \cite{17} and \cite{18} this new preradicals are useful for describing certain kind of intervals of preradicals and characterize left exact preradicals in terms of some classical elements in a lattice for example pseudocomplements. Now starting with a complete modular meet-continuous lattice $A$ consider a monotone function $d\colon A\rightarrow A$ such that $a\leq d(a)$ for all $a\in A$ call this function an \emph{inflator} and let $I(A)$ be the set of all inflators on $A$. $I(A)$ is a complete lattice in the usual way and $I(A)$ have an extra operation namely the \emph{product} of inflators this product is the composition.
With his product $I(A)$ becomes a non-commutative ordered monoid and for two inflators $d$ and $k$ the comparison $d\vee k\leq dk, kd$ holds. Thus in some-way the product of inflators acts like the coproduct of preradicals. This analogy suggest most of the content in the paper. Now we give a brief organization of this: Section \ref{pre} is devoted to recall most of the concepts that are required for the rest of the investigation, in section \ref{sec:sec2} we introduce the \emph{equalizer} and the \emph{totalizer} of an inflator, we prove some properties of these and after that, we specialize in totalizers. Via this we give certain partitions in the set of inflators and we see that this correspond with some intervals of inflators and it is observed that these intervals are the key to give a characterization of \emph{strongly atomic} lattices in our sense. In section \ref{sec3}, we link the totalizers with the concept of dimension of an inflator. In section \ref{sec:sec4} we give some observations related with the uses of totalizers and some insights of the uses of equalizers. 

\section{Preliminaries}
\label{pre}

In this paper we assume that each lattice $(A,\leq,\vee,\wedge,\bar{1},\underline{0})$ is complete, it has a top $\bar{1}$ and a bottom $\underline{0}$, we write $\vee$ and $\wedge$ the binary supremum and infimum operations respectively, on $A$ and $\leq$ for the carried comparison. 
We deal with \emph{modular} lattices, that is, \[(a\vee c)\wedge b=a\vee(c\wedge b)\] for all $a,b,c\in A$ such that $a\leq b$. Now recall that a subset $X\subseteq A$ is directed if it is non-empty and for each $x,y\in X$ there is some $z\in X$ with $x\leq z$ and $y\leq z$.

A lattice $A$ is {\em meet-continuous} if 
\[a\wedge (\bigvee X)=\bigvee\{a\wedge x\mid x\in X\}\leqno({\rm IDL})\] 
for all $a\in A$ and $X\subseteq A$ any directed set; here, for a set $X$, $\bigvee X=\bigvee_{x\in X}x$; similarly for $\bigwedge X$. 
Following Simmons \cite{23} we call a complete, modular, meet continuous lattice $A$ an \emph{idiom}.
Two fundamental examples are the following: 
Given a ring $R$ and any left $R$-module $M$, the lattice $\text{Sub}_{R}(M)$ of all submodules of $M$ is modular and   meet-continuous, hence it is an idiom. Much of the analysis we will describe is inspired by this lattice. We remark that since the lattice of submodules of a given module in general is not a distributive lattice. 

An idiom is a distributive lattice precisely when it is a \emph{frame}, i.e., a complete lattice $A$ that satisfies
\[a\wedge (\bigvee X)=\bigvee\{a\wedge x\mid x\in X\}\leqno({\rm FDL})\]
for all $a\in A$ and $X\subseteq A$ any subset. It follows that each frame is an idiom.
Frames are the algebraic version of a topological space. Indeed if $S$ is a topological space then its topology, $\mathcal{O}(S)$ is a frame, these kind of frames have been extensively studied, for example see \cite{11} and \cite{13}.

Remember that in any lattice $A$ an \emph{implication} in $A$ is a two placed operation $(\_\succ \_)$ given by $x\leq (a\succ b)\Leftrightarrow x\wedge b\leq a$, for all $a,b\in A$. For a proof of the following fact, see \cite{24}.

\begin{prop}\label{03}
A complete lattice $A$ is a frame if and only if $A$ has an implication.
\end{prop}

It follows that in a frame $A$ any element $a\in A$ has a \emph{pseduocomplement} or \emph{negation} $(a\succ \underline{0})$ or simply $\neg a$.
For general background on idioms and frames, see \cite{28}  and \cite{30}.

An \emph{inflator} on an idiom $A$ is a function $d\colon A\rightarrow A$ such that $x\leq d(x)$ and $x\leq y \Rightarrow d(x)\leq d(y)$. 
A \emph{prenucleus} $d$ on $A$ is an inflator such that $d(x\wedge y)=d(x)\wedge d(y)$. 
A \emph{stable} inflator on $A$ is an inflator such that $d(x)\wedge y\leq d(x\wedge y)$ for all $x,y\in A$.   
Let us denote by $I(A)$ the set of all inflators on $A$ and let $P(A)$ be the set of all prenuclei and $S(A)$ the set of all stable inflators. Clearly, $P(A)\subseteq S(A)\subseteq I(A)$. 
Note that from the definition of inflator, the composition of any two inflators is again an inflator. $I(A)$ is a poset with the order given for $d,d'\in I(A)$ by $d\leq d'\Leftrightarrow d(a)\leq d'(a)$, for all $a\in A$. The identity of $A$, denoted by $d_{\underline{0}}$ and the constant function $\bar{d}(a)=\bar{1}$ for all $a\in A$, are inflators (in fact these two are prenuclei). 
 
For an arbitrary idiom $A$ and any subset $\mathcal{I}$ of inflators on $A$, the \emph{infimum} $\bigwedge\mathcal{I}$ of $\mathcal{I}$ is the function on $A$ given by $(\bigwedge\mathcal{I})(a)=\bigwedge\left\{f(a)\mid f\in\mathcal{I}\right\}$ for each $a\in A$. It is immediate that this function is again an inflator, and in fact it is the infimum of the family $\mathcal{I}$. Hence, the poset $I(A)$ is a complete lattice. Moreover:

\begin{enumerate}
\item If $\mathcal{I}\subseteq S(A)$, then $\bigwedge\mathcal{I}\in S(A)$.
\item If $\mathcal{I}\subseteq P(A)$, then $\bigwedge\mathcal{I}\in P(A)$.
\end{enumerate}  
The top of $I(A)$ is $\bar{d}$ and the bottom is $d_{\underline{0}}$. 
For an idiom $A$ and a non-empty subset $\mathcal{S}\subseteq I(A)$, the \emph{supremum} of $\mathcal{S}$ is the function given by $(\bigvee\mathcal{S})(a)=\bigvee\left\{f(a)\mid f\in\mathcal{S}\right\}$, for each $a\in A$. Also, if  $\mathcal{S}$ is directed, then: 

\begin{enumerate}
 \setcounter{enumi}{2}
\item If $\mathcal{S}\subseteq S(A)$, then $\bigvee\mathcal{S}\in S(A)$.
\item If $\mathcal{S}\subseteq P(A)$, then $\bigvee\mathcal{S}\in P(A)$.
\end{enumerate}
We require that $\mathcal{S}\neq\emptyset$ because the supremum of the empty set is not an inflator.  
For two inflators $d$ and $d'$ on $A$ and for all $a\in A$ we have $a\leq d(a)$. From the monotonicity of $d'$ it follows that $d'(a)\leq d'(d(a))$ and we have that $d(a)\leq dd'((a))$. So,  for any two inflators $d,d'\in I(A)$ we have that $d'\vee d\leq d'd$. It follows that if  $A$ is an idiom and $d$, $d'$, $k$ are inflators on $A$, then:

\begin{enumerate}
 \setcounter{enumi}{4}
\item If $d\leq d'$, then $kd\leq kd'$ and $dk\leq d'k$.
\item $kd'\vee kd\leq k(d'\vee d)$ and $ k(d'\wedge d)\leq kd'\wedge kd$. 
\item Moreover, if $\mathcal{D}\subseteq I(A)$ is non-empty, then:\label{item7}
\begin{itemize}
\item $(\bigvee \mathcal{D})k=\bigvee\{dk\mid d\in \mathcal{D}\}$,
\item $(\bigwedge \mathcal{D})k=\bigwedge\{dk\mid d\in \mathcal{D}\}$.
\end{itemize}
\end{enumerate}

Given an inflator $d\in I(A)$, let $d^{0}:=d_{\underline{0}}$,  $d^{\alpha+1}:=d\circ d^{\alpha}$ for a non-limit ordinal $\alpha$, and let  $d^{\lambda}:=\bigvee\{d^{\alpha}\mid \alpha<\lambda\}$ for a limit ordinal $\lambda$.
These are inflators, and from the comparison $d\vee d'\leq dd'$, we have a chain of inflators 
\[d\leq d^{2}\leq d^{3}\leq\ldots \leq d^{\alpha}\leq\ldots.\] 
By a cardinality argument, there exists an ordinal $\gamma$ such that $d^{\alpha}=d^{\gamma}$, for $\alpha\geq \gamma$. In fact we can choose the least of these ordinals, say $\infty$. Thus, $d^{\infty}$ is an inflator such that $d\leq d^{\infty}$ but more important, this inflator satisfies $d^{\infty}d^{\infty}=d^{\infty}$, that is, $d^{\infty}$ is an idempotent or a closure operator on $A$. Denote by $C(A)$ the set of all closure operators. This is a poset in which the infimum of a set of closure operators is again a closure operator. Hence, it is a complete lattice, and   the construction we have just made defines an operator $(\_)^{\infty}\colon I(A)\rightarrow C(A)$. It is clear that this operator (in the second level) is inflatory and monotone. The supremum  in $C(A)$ of an arbitrary family of closure operators can be computed as follows:  first, take a non-empty subset of closure operators $\mathcal{C}$ 
on $A$.  The supremum, $\bigvee \mathcal{C}$ is an inflator so we can apply the construction above and obtain a closure operator $(\bigvee \mathcal{C})^{\infty}$;  this is the supremum of the family $\mathcal{C}$ in $C(A)$.  For an inflator $j$ on an idiom $A$, we say that $j$ is a \emph{nucleus} if it is an idempotent pre-nucleus. By induction and idiom distributivity it follows that if $A$ is an idiom, then: 
\begin{enumerate}
 \setcounter{enumi}{7}
\item If $f$ is a pre-nucleus on $A$, then $f^{\alpha}$ is a pre-nucleus for all ordinal $\alpha$; in particular $f^{\infty}$ is a nucleus.
\item If $f$ is a stable inflator, then each $f^{\alpha}$ is a stable for all ordinal $\alpha$, and for limit ordinals $\lambda$, $f^{\lambda}$ is a pre-nucleus; in particular $f^{\infty}$ is a nucleus.
\end{enumerate}

The set $N(A)$ of all nuclei on $A$ is a poset, and the infimum of a family of nuclei is again a nucleus; thus it is a complete lattice. For the supremum of any family of nuclei we require the constructions above and item 4 in the list before: first we take a family of nuclei $\mathcal{N}$, then we consider the family of all compositions of elements of $\mathcal{N}$, say $\mathcal{N}^{\circ}$; this is a direct family, hence $\bigvee \mathcal{N}^{\circ}$ is the supremum in $P(A)$;  then we apply the $\infty$ construction, and this new inflator $(\bigvee \mathcal{N}^{\circ})^{\infty}$ is a nucleus. In fact, it is the supremum of the family $\mathcal{N}$. 
 

Each nucleus $j$ determines a quotient of $A$ given by the set of fixed points of $j$ in $A$, that is, $A_{j}=\{x\in A\quad j(x)=x\}$. 
There is a surjective idiom morphism $j^{*}\colon A\rightarrow A_{j}$ given by $j^{*}(a)=j(a)$.  When the idiom $A$ is a frame then $A_{j}$ is a frame (in this way one constructs sublocales of a locale see \cite{11}).
We will need the following:
\begin{thm}\label{014}
For any idiom $A$, the complete lattice $N(A)$ is a frame.
\end{thm}

For more details, we refer the reader to \cite{11}, \cite{24} and \cite{26}.

There is another approach to congruences for an idiom, we will recall some of that material. In \cite{23} and \cite{27} the author describes a \emph{module theoretic like} technique to construct inflators, stable inflators, prenuclei and nuclei, via the \emph{base frame} of the idiom $A$.
To construct the base frame of an idiom $A$ one needs the set of all intervals of $A$, that is, given $a,b\in A$ such that $a\leq b$, the \emph{interval} $[a,b]$ is the set $[a,b]=\{x\in A\mid a\leq x \leq b\}$.  Denote by $\EuScript{I}(A)$ the set of all intervals of $A$. Now to compare two intervals $I, J$, we say that $I$ is a \emph{subinterval} of $J$, denoted by $I\rightarrow J$,  if $I\subseteq J$, that is, if $I=[a,b]$ and $J=[a',b']$ with $a'\leq a\leq b\leq b'$ in $A$. We say that $J$ and $I$ are \emph{similar}, denoted by  $J\sim I$, if there are $l,r\in A$ with associated intervals \[L=[l,l\vee r]\quad [l\wedge r,r]=R\] where $J=L$ and $I=R$ or $J=R$ and $I=L$. Clearly, this a reflexive and symmetric relation. Moreover, if $A$ is modular, this relation is just the canonical lattice isomorphism between $L$ and $R$. Now we are going to impose some closure proprieties in sets of intervals:

We say that a set of intervals $\mathcal{A}\subseteq {\EuScript I}(A)$ is \emph{abstract} if is not empty and it is closed under $\sim$, that is, 
\[J\sim I\in\mathcal{A}\Rightarrow J\in\mathcal{A}.\] 
An abstract set $\mathcal{B}$ is a \emph{basic} set of intervals if it is closed by subintervals, that is, 
\[J\rightarrow I\in\mathcal{B}\Rightarrow J\in\mathcal{B}\] 
for all intervals $I,J$. A set of intervals $\mathcal{C}$ is a \emph{congruence} set if it is basic and closed under abutting intervals, that is, 
\[[a,b][b,c]\in \mathcal{C}\Rightarrow [a,c]\in\mathcal{C}\]
for elements $a,b,c\in A$.  A basic set of intervals $\mathcal{B}$ is a \emph{pre-division} set if \[\forall x\in X[a,x]\in\mathcal{B}\Rightarrow [a,\bigvee X]\in\mathcal{B}\] for each $a\in A$ and $X\subseteq [a,\bar{1}]$. A set of intervals $\mathcal{D}$ is a \emph{division} set if it is a congruence set and a pre-division set.
Put $\EuScript{D}(A)\subseteq\EuScript{C}(A)\subseteq\EuScript{B}(A)\subseteq\EuScript{A}(A)$ the set of all division, congruence, basic and abstract set of intervals in $A$. 

Note that $\EuScript{B}(A)$ is closed under arbitrary intersections and unions, hence it is a frame. The top of this frame is $\EuScript{I}(A)$ and the bottom is the set of all trivial intervals of $A$, denoted by $\EuScript{O}(A)$ or simply by $\EuScript{O}$. The frame $\EuScript{B}(A)$ is the \emph{base} frame of the idiom $A$. 

The family $\EuScript{C}(A)$ is closed under arbitrary intersections, but suprema are not unions; to correct this we take in any basic set $\mathcal{B}$ the least congruence set that contains it. From this one can show that $\EuScript{C}(A)$ is a frame. 

For the set $\EuScript{D}(A)$ and for any $\mathcal{B}\in\EuScript{B}(A)$ we can describe the least division set that contains it. Since $\EuScript{D}(A)$ is closed under arbitrary intersections, denote by $\EuScript{D}vs(\mathcal{B})$ that division set that contains it. In \cite{23} it is proved that $\EuScript{D}vs(\mathcal{B})$ is a nucleus on $\EuScript{B}(A)$ and the quotient of this nucleus is $\EuScript{D}(A)$. In fact, there is a relation with this frame and the frame $N(A)$: To describe this relation, take any basic set $\mathcal{B}$ and $a\in A$; define $|\mathcal{B}|(a)=\bigvee X$, where $x\in X\Leftrightarrow [a,x]\in\mathcal{B}$. This produces the associated inflator of $\mathcal{B}$. Moreover, if the basic set $\mathcal{B}$ is a congruence set, then $|\mathcal{B}|$ is a pre-nucleus in $A$, and if it is a division set, then $|\mathcal{B}|$ is a nucleus. In this way we have  for every division set  a nucleus. Now, given a nucleus $j$ we can construct a division set $[a,b]\in\mathcal{D}_{j}\Leftrightarrow j(a)=j(b)$.  This correspondences are bijections and moreover they define an isomorphism between $\EuScript{D}(A)$ and $N(A)$. The details of all these  are in \cite{23}, and a more recent account is given in \cite{27} and \cite{28}. 

There are interesting examples of all this: An interval $[a,b]$ is \emph{simple} if there is no $a< x< b$ that is $[a,b]=\left\{a,b\right\}$. Denote by $\EuScript{S}mp$ be the set of all simple intervals. An interval $[a,b]$ of $A$ is \emph{complemented} if it is a complemented lattice, that is, for each $a\leq x\leq b$ there exist $a\leq y\leq b$ such that $a=x\wedge y$ and  $b=x\vee y$.  Let $\EuScript{C}mp$ be the set of all complemented intervals. In fact, for every $\mathcal{B}$ we can define $\EuScript{S}mp(\mathcal{B})$ and $\EuScript{C}mp(\mathcal{B})$: the former is the set of intervals that are $\mathcal{B}$-simple, that is, the set of all $[a,b]$ such that for each $a\leq x\leq b$, $[a,x]\in\mathcal{B}$ or $[x,b]\in\mathcal{B}$, and the latter is the set of all intervals that are $\mathcal{B}$-complemented, that is, $[a,b]$ such that for every $a\leq x\leq b$ exists $a\leq y\leq b$ such that $[a,x\wedge y]\in\mathcal{B}$ and $[x\vee y,b]\in\mathcal{B}$.  With this, we have that $\EuScript{S}mp=\EuScript{S}mp(\EuScript{O})$ and $\EuScript{C}mp=\EuScript{C}mp(\EuScript{O})$.

In fact one can generalize the above sets of intervals: Given any $\mathcal{B}\in\EuScript{B}(A)$ denote by $\EuScript{C}rt(\mathcal{B})$ the set of intervals such that for all $a\leq x\leq b$ we have $a=x$ or $[x,b]\in\mathcal{B}$; this is the set of all $\mathcal{B}$-\emph{critical} intervals. Note that $\EuScript{S}mp(\EuScript{O})=\EuScript{C}rt(\EuScript{O})$ and for any $\mathcal{B}\in\EuScript{B}(A)$, $\EuScript{C}rt(\mathcal{B})\leq \EuScript{S}mp(\mathcal{B})$.

We are interested in the \emph{simple like} intervals: 

A interval $[a,b]$ is \emph{atomic} if for each $a<d\leq b$ exists $a<z\leq d$ such that $[a,z]\in\EuScript{S}mp$.
A interval $[a,b]$ is \emph{strongly atomic} if each subinterval $a\leq c< d\leq b$ is atomic, that is, exists $c<z\leq d$ with $[c,z]\in\EuScript{S}mp$ and we say that the idiom $A$ is \emph{strongly atomic} if each interval is strongly atomic. Let $\EuScript{S}\EuScript{A}$ be the set of all strongly atomic intervals in $A$.

Now if we set $soc=|\EuScript{S}mp|$ and $cdb=|\EuScript{C}mp|$ call this inflators the \emph{socle derivative} and the \emph{Cantor-Bendixson derivative} of the idiom $A$, this are the fundamental inflators associated to every idiom and one can see that $soc$ and $cbd$ are stable inflators on $A$ (6.13 and 6.17 of \cite{27}), even more if the idiom $A$ is a frame the $cdb$ it is know that this inflator is a pre-nucleus. We know that $soc^{\infty}$ is a nucleus so by the assignment mentioned before, it corresponds a division set $\mathcal{D}\in\EuScript{D}(A)$, this division set is exactly the set $\EuScript{S}\EuScript{A}$ (7.11 of \cite{27}).

%

\section{Operators in $I(A)$}
\label{sec:sec2}

In \cite{17} the authors introduce four operators in the (big) lattice $R$-pr. They associate two operators to the product of preradicals $\tau\cdot \sigma$, and two operators to the coproduct of preradicals $(\tau:\sigma)$. For a general lattice $A$ we introduce two operators, $t(\_)$ and $e(\_)$ on $I(A)$.

\begin{dfn}\label{i1}
Let $d\in I(A)$ an inflator we define:
\[\mathcal{I}_{e}(d)=\{z\in I(A)\mid zd=d\}\] and \[\mathcal{I}_{t}(d)=\{z\in I(A)\mid zd=\bar{d}\}\]
and we call them the set of \emph{equalizers} of $d$ and the set of \emph{totalizers} of $d$, respectively. Note that $\mathcal{I}_{e}(d)$ and $\mathcal{I}_{t}(d)$ are non-empty, since $d_{\underline{0}}\in\mathcal{I}_{e}(d)$ and $\bar{d}\in \mathcal{I}_{t}(d)$. We can consider
\[\bigvee \mathcal{I}_{e}(d):=e(d)\] and \[\bigwedge \mathcal{I}_{t}(d):=t(d)\] 
 and we call these inflators the \emph{equalizer} of $d$, and the \emph{totalizer} of $d$, respectively. Note that from the definition and item \ref{item7} of Section \ref{sec:sec2}, we have that $t(d)d=\bar{d}$ and $e(d)d=d$. 
\end{dfn}

\begin{lem}\label{i2}
Let $A$ be any idiom, $d,d'\in I(A)$ any two inflators, and $\mathcal{I}$ a non-empty family of inflators.  Then:

\begin{enumerate}
\item If $d\leq d'$ then $t(d')\leq t(d)$.

\item $t(d_{\underline{0}})=\bar{d}$ and $t(\bar{d})=d_{\underline{0}}$.

\item $t(\bigvee \mathcal{I})\leq \bigwedge \{t(d)\mid d\in\mathcal{I}\}$.

\item $t(\bigwedge \mathcal{I})\geq \bigvee\{ t(d)\mid d\in\mathcal{I}\}$.

\item $e(d)$ is an idempotent inflator.

\item $e(d)\leq d$.

\item $e(d)=d$ if and only if $d$ is idempotent.
\end{enumerate}
\end{lem}

\begin{proof}
The properties concerning the totalizers are straightforward.

For 5, we have $e(d)\leq e(d)e(d)$ and $[e(d)e(d)]d=e(d)[e(d)d]=e(d)d=d$, that is, $e(d)e(d)\leq e(d)$. Now note that for each $z\in\mathcal{I}_{e}(d)$, from the comparison of supremum and the product of inflators, we have that  
$z\leq d$. 

Finally, 6 and 7 follow directly from 5 and the fact that if $d$ is idempotent then $d^{2}=d$, and so $d\leq e(d)$.

\end{proof}

\begin{obs}\label{remark6}
Notice that the operator $e:I(A)\rightarrow C(A)$ and $t:I(A)\rightarrow I(A)$, is not monotone.
\end{obs}

Now, fix $b\in A$ and let $a\in A$. We define:
\begin{displaymath}
O_b(a)=\begin{cases}
  \bar{1} &  \text{if $a\geq b$} \\
a & \text{otherwise}.
\end{cases}
\end{displaymath}
Clearly $O_{b}$ is an inflator.

\begin{thm}\label{i3}
Let $A$ be an idiom. Then, 
\[O_{d(\underline{0})}=t(d)\] for each $d\in I(A)$.
\end{thm} 

\begin{proof}
Take any inflator $d$ on the idiom $A$ and consider its value $d(\underline{0})$,  and then take the inflator $O_{d(\underline{0})}$. Then, $O_{d(\underline{0})}d=\bar{d}$, because $d(\underline{0})\leq d(a)$ for any $a\in A$ and hence $t(d)\leq O_{d(\underline{0})}$. 

On the other hand, consider 
\[h=\bigwedge\{z\in I(A)\mid z(d(\underline{0}))=\bar{1}\}\] 
and note that this $h$ totalizes $d$. Therefore, $t(d)\leq h$, but clearly $h\leq t(d)$ so $h=t(d)$.  Now observe that for any $a\in A$, by the definition of $O_{d(\underline{0})}$ we have that for an inflator $z$ such that $zd(\underline{0})=\bar{d}$ then $O_{d(\underline{0})}(a)\leq z(a)$ , thus $O_{d(\underline{0})}\leq h$, that is, $h=O_{d(\underline{0})}$.

\end{proof}

Denote by $\text{\rm Tot}(I(A))$ the set of all the inflators of the form $O_{d(\underline{0})}$. Observe that this set is a poset. 

\begin{prop}\label{i4}
Let $A$ be an idiom, and $\text{\rm Tot}(I(A))$ the set of all totalizers, let $\mathcal{G}$ be a non-empty family of totalizers with $\mathcal{D}$ its set of associated inflators then:

\begin{enumerate}
\item $O_{\bigwedge\mathcal{D}(\underline{0})}$ is the supremum of the family $\mathcal{G}$.
\item $O_{\bigvee\mathcal{D}(\underline{0})}$ is the infimum fo the family $\mathcal{G}$.
\end{enumerate}
\end{prop}
\begin{proof}
Take a non-empty family of these inflators $\mathcal{G}$ and let $\mathcal{D}$ be the set of inflators given by $d\in\mathcal{D}\Leftrightarrow O_{d(\underline{0})}\in\mathcal{G}$. Then, consider the supremum , $\bigvee\mathcal{D}$ of $\mathcal{D}$. This inflator has  totalizer $O_{\bigvee\mathcal{D}(\underline{0})}$, and by  construction we have that $O_{\bigvee\mathcal{D}(\underline{0})}\leq O_{d(\underline{0})}$ for every $d\in\mathcal{D}$. Note that if there is a totalizer $O_{z(\underline{0})}$ such that $O_{z(\underline{0})}\leq O_{d(\underline{0})}$, for all $d\in\mathcal{D}$, then $d(\underline{0})\leq z(\underline{0})$ for all $d\in\mathcal{D}$. Thus, $\bigvee\mathcal{D}(\underline{0})\leq z(\underline{0})$ and from this we conclude that $O_{\bigvee\mathcal{D}(\underline{0})}\leq O_{z(\underline{0})}$, that is, $O_{\bigvee\mathcal{D}(\underline{0})}$ is the infimum of the family $\mathcal{G}$ in $\text{\rm Tot}(I(A))$. For the description of the supremum of an arbitrary family of totalizers we proceed symmetrically.
\end{proof}

\begin{obs}\label{remark2.6}
From the description of  totalizer, $t(t(d))=\bar{1}$. We also note that $(z\vee O_{a})(b)=z(b)\vee O_{a}(b)$, from where it follows that the supremum $(z\vee O_{a})(b)$ is $\bar{1}$ or $z(b)$.  Therefore, this supremum is actually $zO_{a}$ for any inflator $z$. These observations tell us that the product of two totalizers is commutative, as inflators. Another important observation is that $\text{\rm Tot}(I(A))$ satisfies a reversed idiom distributivity law of $A$, and if $A$ is a frame then $\text{\rm Tot}(I(A))$ also satisfies a reversed frame distributivity law. 
\end{obs}

Given any $a\in A$ define
\begin{displaymath}
\iota_{a}(b)=
\begin{cases}
a & \text{if }b=\underline{0} \\
\bar{1} & \text{if }b \neq \underline{0}.
\end{cases}
\end{displaymath}
This is clearly an inflator. Moreover the assignment $a\mapsto \iota(a)$ is an embedding from $A$ to $I(A)$. Now, note that $\iota_{a}\iota_{b}=\bar{d}$, and for any inflator $d$, $d\iota_{a}=\iota_{d(a)}$, and then $d\iota_{\underline{0}}=\iota_{d(\underline{0})}$. By definition of these inflators we have that $d\leq \iota_{d(\underline{0})}$, in particular $\iota_{d(\underline{0})}d=\bar{d}$.  Then, by construction of the totalizer $t(d)\leq \iota_{d(\underline{0})}$.

To proceed with the study of totalizers as in \cite{17}, we following relation in $I(A)$:

Given two inflators $d$ and $d'$ we say that $d\sim_{t} d'\Leftrightarrow t(d)=t(d')$. This is clearly an equivalence relation.  Denote by $[d]_{t}$ an equivalence class of this relation. 

Observe that $d\in[d_{\underline{0}}]_{t}\Leftrightarrow t(d)=t(d_{\underline{0}})=\bar{d}$, and $d\in[\bar{d}]_{t}\Leftrightarrow t(d)=d_{\underline{0}}$, that is, $d=d_{\underline{0}}d=\bar{d}$. Hence, $[\bar{d}]_{t}=\{\bar{d}\}$.

\begin{prop}\label{i5}
If $A$ is any idiom and $d$ an inflator on $A$, then $[d]_{t}$ is an interval of the form \[[\bigwedge\mathcal{T}_{d},\bigvee\mathcal{T}_{d}]\] 
Moreover, there is a bijective correspondence between $\text{\rm Tot}(I(A))$ and the set of these intervals.
\end{prop}

\begin{proof}
Now, let $\mathcal{T}_{d}=[d]_{t}$ and consider the supremum $\bigvee\mathcal{T}_{d}$ and  infimum $\bigwedge\mathcal{T}_{d}$ of this family.  By the above description of the totalizers, we know that the totalizers of these two inflators are $O_{(\bigvee\mathcal{T}_{d})(\underline{0})}$ and $O_{(\bigwedge\mathcal{T}_{d})(\underline{0})}$. But we also know that these two inflators are the infimum and the supremum of the family of totalizers $O_{d'(\underline{0})}$ with $d'\in\mathcal{T}_{d}$. This means that $O_{(\bigvee\mathcal{T}_{d})(\underline{0})}=\bigwedge\{O_{d'(\underline{0})}\mid d'\in\mathcal{T}_{d}\}$ and $O_{(\bigwedge\mathcal{T}_{d})(\underline{0})}=\bigvee\{O_{d'(\underline{0})}\mid d'\in\mathcal{T}_{d}\}$.  But $\bigvee\{O_{d'(\underline{0})}\mid d'\in\mathcal{T}_{d}\}=O_{d(\underline{0})}$ and $\bigwedge\{O_{d'(\underline{0})}\mid d'\in\mathcal{T}_{d}\}=O_{d(\underline{0})}$. This shows that $\bigvee\mathcal{T}_{d}$ and $\bigwedge\mathcal{T}_{d}$ are in $[d]_{t}$.

For the bijection, observe that $[d]_{t}=[d']_{t}$ if and only if $d(\underline{0})=d'(\underline{0})$ if and only if $O_{d(\underline{0})}=O_{d'(\underline{0})}$.
\end{proof}

\begin{thm}\label{i6}
Let $A$ be an idiom and $d$ any inflator on $A$. Then the intervals of the form $[u_{d(\underline{0})},\iota_{d(\underline{0})}]$ are collapsed by taking totalizers. Moreover, these intervals are just the blocks of the partition $\sim_{t}$, and there is a bijection between $\text{\rm Tot}(I(A))$ and the set of intervals of the form $[u_{d(\underline{0})},\iota_{d(\underline{0})}]$.
\end{thm}

\begin{proof}
Take any element $a\in A$ with $A$ an idiom, and set $u_{a}(b)=a\vee b$. It is immediate that $u_{a}$ is an idempotent inflator on $A$ (if $A$ is a frame this inflator is a nucleus). Now suppose we have an inflator $d$ on $A$ and  consider $d(\underline{0})$ and  $u_{d(\underline{0})}$.  By the construction of the totalizer we know that $t(u_{d(\underline{0})})=O_{u_{d(\underline{0})}(\underline{0})}$. Now, since $u_{d(\underline{0})}(\underline{0})=d(\underline{0})$, then $t(u_{d(\underline{0})})=t(d)$.  For the inflator $\iota_{d(\underline{0})}$, its  totalizer is just $O_{d(\underline{0})}$. Hence, from the definitions we have $d\leq \iota_{d(\underline{0})}$ and $u_{d(\underline{0})}\leq d$.

By last Proposition we have the bijection.
\end{proof}

\begin{ej}\label{ex1}
For any associative ring with unit $R$ consider the category $R$-$\Mod$ of left $R$-modules. Let $R$-$\tors$ be the frame of all hereditary torsion theories. Most of the interplay between the structure of the ring $R$ and the category $R$-$\Mod$ is related to this frame. See \cite{8} and \cite{30}. We are interested on the various notions of dimension associated to the category $R$-$\Mod$ as in \cite{1}, \cite{2} and \cite{27}. Now recall that for $\tau\in R-\tors$ a module $M$ is $\tau$\emph{-cocritical} if $M$ is $\tau$-torsion free and $M/N$ is  $\tau$-torsion for every $0\neq N\subseteq M$. $M$ is said \emph{cocritical} if it is $\tau$-cocritical for some $\tau\in R$-$\tors$. Put $R$-sp=$\left\{\chi(M)\mid M \text{ is cocritical} \right\}$.

Let $\mathfrak{g}:R$-$\tors\to R$-$\tors$ given by
\[\mathfrak{g}(\tau)=\tau\vee(\bigvee\{\xi(N)\mid N \;\text{is $\tau$-cocritical}\}).\]
This is an inflator in $R$-$\tors$. Moreover, it is not hard to see that $\mathfrak{g}$ is a pre-nucleus, in fact if $\mathfrak{g}^{\infty}=\bar{d}$ then $R$ has Gabriel dimension. The next two propositions characterize $\mathfrak{g}(\tau)$:

\begin{prop}\label{i7}
Let $R$ be an associative ring with unit. The next statements are equivalent:
\begin{enumerate}
\item  $u_{\mathfrak{g}(\xi)}=\mathfrak{g}$.
\item $R\text{\rm -sp}=\{\chi(S)\mid S\in R\text{\rm -simp}\}$, where $R\text{\rm-simp}$ is a complete set of representatives of isomorphism classes of simple $R$-modules.
\end{enumerate}
\end{prop}

\begin{proof}
$(1)\Rightarrow(2)$: 
Let $\tau\in R$-sp, then there is an $M$, $\tau$-cocritical such that $\chi(M)=\tau$. By the definition of $\mathfrak{g}$, we have: 
$$\tau\vee(\bigvee\{\xi(N)\mid N \;\text{is $\tau$-cocritical}\})=\mathfrak{g}(\chi(M))=\chi(M)\vee \mathfrak{g}(\xi)=\chi(M)\vee( \vee_{S\in R-simp} \xi(S)).$$ 
Then, there exists $S$ simple such that $t_{\xi(s)}(M)\neq 0$ and so $S\leq M$. Now,  since $S$ is essential in $M$, then $E(S)=E(M)$, that is, $\chi(S)=\chi(M)$.

$(2)\Rightarrow(1)$: 
Again take any proper torsion theory $\tau$, and let $M$ be a $\tau$-cocritical module.  Then $\chi(M)\in R$-sp, and by  $(2)$ there exist a simple $S$ such that $\chi(M)=\chi(S)$. Hence, $S\in \mathcal{F}_{\tau}$ and so $S$ is $\tau$-cocritical.  Therefore, $\tau\vee \xi(M)=\tau\vee\xi(S)$. Notice that a simple $R$-module is $\tau$-torsion or $\tau$-torsionfree, hence $u_{\mathfrak{g}(\xi)}=\mathfrak{g}$.

\end{proof}

Recall that in a left semiartinian ring $R$, every hereditary torsion theory $\xi\neq\tau\in R$-tors is a supremum of atoms.

\begin{cor}\label{i8}
For any ring $R$ the following are equivalent.:
\begin{itemize}
\item[{\rm (1)}]  $R$ is left semiartinian.
\item[{\rm (2)}] $\mathfrak{g}=\iota_{\mathfrak{g}(\xi)}$.
\item[{\rm (3)}] $u_{\mathfrak{g}(\xi)}=\mathfrak{g}=\iota_{\mathfrak{g}(\xi)}$
\end{itemize}

\end{cor}
\end{ej}
Now we give a generalization of the above result into the idiomatic case, for that purpose we need to recall some facts about certain kind of nuclei associated to any interval on an idiom $A$. 
In \cite{23} and \cite{25} the author introduces the following inflator, given an interval $[a,b]$ in an idiom $A$, consider the set $\mathcal{F}=\{f\in P(A)\mid f(a)\wedge b=a\}$, this is non-empty and in \cite{23} Lemma 4.3 or Lemma 5.1 of \cite{25} it is proved that $\mathcal{F}$ satisfy:
\begin{enumerate}
\item $\mathcal{F}$ is directed.
\item $\bigvee\mathcal{F}\in\mathcal{F}$.
\item $\bigvee\mathcal{F}$ is a nucleus.
\end{enumerate}
Name this nucleus $\chi(a,b)$, from the above $j\leq\chi(a,b)\Leftrightarrow j(a)\wedge b=a$ for all $j\in N(A)$. This nucleus is the idiomatic analogue of the torsion theory cogenerated by a module $M$. This nucleus gives some interesting intervals.
\begin{dfn}\label{i91}
An interval $[a,b]$ with $a<b$ is \emph{inert} if for any $a<x\leq b$ then $\chi(a,x)=\chi(a,b)$
\end{dfn} 

\begin{dfn}
An interval $[a,b]$ is \emph{uniform} if it is non-trivial and if $x\wedge y=a$ then $x=a$ or $y=a$ for all $x,y\in [a,b]$.
\end{dfn}

\begin{obs}
\begin{enumerate}
	\item For each inert interval $[a,b]$, the nucleus $\chi(a,b)$ is a point (a $\wedge$-irreducible element) in $N(A)$.
	\item All uniform interval are inert.
	\item Let $\mathcal{C}$ be a congruence set in $A$. If $[a,b]\in\EuScript{C}rt(\mathcal{C})-\mathcal{C}$ then $[a,b]$ is uniform and hence inert.
\end{enumerate}
\end{obs}

\begin{obs}
 Given an interval $[a,b]$ on $A$ by the completeness of $\EuScript{D}(A)$ we can consider the least division set that contain the interval $[a,b]$, denote this division set by $\mathcal{D}(a,b)$ and the corresponding nucleus $\xi(a,b)\in N(A)$ then $\xi(a,b)\leq k\Leftrightarrow b\leq k(a)$.
\end{obs}


We can consider the inflators $\EuScript{D}vs$, $\EuScript{C}rt$ on the base frame $\EuScript{B}(A)$. Denote $\EuScript{G}ab=\EuScript{D}vs\circ \EuScript{C}rt$. We know that for every nucleus $j$ on $A$ correspond a unique division set on $A$, $\mathcal{D}_{j}$. Consider the division set $\EuScript{G}ab(\mathcal{D}_{j})=\mathcal{D}_{k}$ for some $k\in N(A)$. This defines a pre-nucleus on $N(A)$, denote this two pre-nuclei by $Gab$.

In in \cite{28} Theorem 3.7 is probed next description of $Gab$:

\begin{thm}\label{i92}
Let $A$ be an idiom and consider any nucleus $j\in N(A)$. Then \[Gab(j)=\bigvee\{\xi(a,b)\mid [a,b]\in\EuScript{C}rt(\mathcal{D}_{j})\}\]
\end{thm}


Following Simmons (Definition 7.8 of \cite{28}), we say that a point $\pi\in N(A)$ is a \rm{$G$}\emph{-point} if $\pi<Gab(\pi)$.
With all this we can perform the idiomatic version of Proposition \ref{i7} and Corollary \ref{i8}.

\begin{prop}\label{i81}
The following are equivalent for an idiom $A$
\begin{itemize}
\item[(1)] $u_{Gab(d_{\underline{0}})}=Gab$.
\item[(2)] \rm{Gpt}$(N(A))=\{\chi(a,b)\mid[a,b]\in\EuScript{S}mp\}$
Here \rm{Gpt}$(N(A))$ denote the $G$-points of $N(A)$ given by $j$-critical interval for some nucleus $j$.
\end{itemize}
\end{prop}

\begin{proof}
Suppose $(1)$, and consider any $[a,b]\in\EuScript{C}rt(\mathcal{D}_{j})-\mathcal{D}_{j}$ for some nucleus $j$ on $A$, then $[a,b]$ is inert so $\chi(a,b)$ is a point in $N(A)$, in fact this point is a $G$-point, that is, $Gab(\chi(a,b))\neq \chi(a,b)$, and for $(1)$ $Gab(\chi(a,b))=\bigvee\{\xi(x,y)\mid[x,y]\in \EuScript{C}rt(\mathcal{D}_{\chi(a,b)})\}=\chi(a,b)\vee \bigvee\{\xi(x,y)\mid[x,y]\in\EuScript{S}mp\}$ hence, there is an interval $[x,y]\in\EuScript{S}mp$ such that $\xi(x,y)$ do not collapse the interval $[a,b]$, then $a\leq\xi(x,y)(a)\wedge b\leq b$ therefore $[a,\xi(x,y)(a)\wedge b]\in\mathcal{D}_{\xi(x,y)}$ and by construction, this division set is $\EuScript{D}vs(\mathcal{B}(x,y))$. Then we can find a proper sub-interval of $[a,\xi(x,y)(a)\wedge b]$ similar to $[x,y]$ (therefore it is a subinterval of $[a,b]$), because the interval $[x,y]$ is simple.
Now if we have $\chi(a,b)(x)\wedge y=y$ then $\xi(x,y)\leq\chi(a,b)$, which is a contradiction of the above argument. Thus one necessarily have $j\leq\chi(a,b)\leq \chi(x,y)$.
To see the other comparison suppose $a<\chi(x,y)(a)\leq b$, since $[a,b]$ is $j$-critical, we obtain $[\chi(x,y)(a)\wedge b,b]\in\mathcal{D}_{j}$, that is, $b\leq j(\chi(x,y)(a))\leq b$ thus $j(b)\leq \chi(x,y)(a)$ (by $j\leq \chi(x,y)$ and both are idempotent) but this implies that the interval $[a,b]$ is collapsed by $\chi(x,y)$ and consequently $[x,y]$ so is, which is a contradiction to the property of $\chi(x,y)$, therefore $\chi(x,y)(a)\wedge b=a$, that is, $\chi(x,y)=\chi(a,b)$.

If we suppose $(2)$, take any $[a,b]\in\EuScript{C}rt(\mathcal{D}_{j})-\mathcal{D}_{j}$ with $j$ a nucleus different of $\bar{d}$, this interval is uniform hence inert so $\chi(a,b)$ is a point in $N(A)$, then by $(2)$ exists a simple interval $[x,y]$ such that $\chi(a,b)=\chi(x,y)$, from this we have $j\leq\chi(x,y)$. Observe that for the nucleus $\xi(a,b)$ we have $x\leq\xi(a,b)(x)\wedge y\leq y$, by simplicity of the interval and the construction of $\xi(a,b)$ we conclude that $\xi(a,b)(x)\wedge y=y$, that is, $[x,y]$ is collapsed by this nucleus, therefore $\xi(x,y)\leq\xi(a,b)$. If we have $\xi(x,y)(a)\wedge b<b$ then $[\xi(x,y)(a)\wedge b,b]\in\mathcal{D}_{j}$ because $[a,b]\in\EuScript{C}rt(\mathcal{D}_{j})$, thus $b\leq j(\xi(x,y)(a))$, that is, $b\leq j\vee \xi(x,y)(a)$ then $\xi(a,b)\leq j\vee\xi(x,y)$, but $\xi(x,y)\leq\xi(a,b)$, therefore $j\vee\xi(x,y)\leq j\vee\xi(a,b)$ and then $j\vee\xi(x,y)=j\vee\xi(a,b)$, which is the same that $Gab(j)=u_{Gab(d_{\underline{0}})}(j)$.
\end{proof}

As a corollary of this we have:

\begin{cor}
\label{i911}

The following are equivalent for an idiom $A$
\begin{itemize}

\item[(1)]  $A$ is $\EuScript{S}\EuScript{A}$.

\item[(2)] $u_{Gab(d_{\underline{0}})}=Gab=\iota_{Gab(\underline{0})}$.

\item[(3)] $Gab=\iota_{Gab(d_{\underline{0}})}$
\end{itemize}
\end{cor}

\section{Dimensions of inflators and totalizers}\label{sec3}

Take any $s\in S(A)$ and consider the set $\mathcal{S}_{t}(s)=\{s'\in S(A)\mid s's=\bar{d}\}$.  Then, $\mathcal{S}_{t}(s)\subseteq \mathcal{I}_{t}(s)$ (Definition \ref{i1}). Hence, $t(s)\leq\bigwedge\mathcal{S}_{t}(s)=\mathfrak{t}(s)$. This new stable inflator $\mathfrak{t}(s)$ is the partial totalizer of $s$ in $S(A)$. This construction can be applied to any pre-nucleus and any nucleus: set $\mathfrak{t}(f)=\bigwedge\mathcal{P}_{t}(f)=\bigwedge\{k\in P(A)\mid kf=\bar{d}\}$ and $\jmath(j)=\bigwedge\mathcal{N}_{t}(j)=\bigwedge\{k\in N(A)\mid kj=\bar{d}\}$. From this and item 9 in Section \ref{pre} we have that $(\mathfrak{t}(s^{\infty}))^{\infty}=\jmath(s^{\infty})$ for any stable inflator $s$. Note also that these observations give us a chain of  inflators $t(s^{\infty})\leq t(s)\leq \mathfrak{t}(s)\leq(\mathfrak{t}(s))^{\infty}$.

We know that if $j$ is a nucleus on $A$, then $A_{j}$ is an idiom, so if we take an inflator $d^{A_{j}}$ on $A_{j}$ we have a diagram \[\xymatrix{ A\ar[r]^{j^{*}} & A_{j} \ar[r]^{d^{A_{j}}} & A_{j}\ar[r]^{j_{*}} & A}\] where $j^{*}(a)=j(a)$ and $j_{*} $ is the inclusion, then $j_{*}\cdot d^{A_{j}}\cdot j^{*}$ is an inflator on $A$. From $j^{*}j_{*}=id_{A}$ we get $(j_{*}t(d^{A_{j}})j^{*})(j_{*}d^{A_{j}}j^{*})=j_{*}t(d^{A_{j}})d^{A_{j}}j^{*}=j_{*}\bar{d}^{A_{j}}j^{*}=\bar{d}$. With this the following is straightforward.

\begin{prop}\label{i9}
Let $A$ be an idiom and $j$ any nucleus. For any $d^{A_{j}}\in I(A_{j})$ we have that
\[t(j_{*}d^{A_{j}}j^{*})\leq j_{*}t(d^{A_{j}})j^{*}.\]
\end{prop}

Now, take any inflator $k$ and define $\mu^{k}\colon I(A)\rightarrow I(A)$ by $\mu^{k}(d)=dk$. Then, from item \ref{item7} of Section \ref{pre}, $\mu^{k}$ is a pre-nucleus. If the inflator $k$ is idempotent, then $\mu^{k}$ is a nucleus on $I(A)$. Moreover, if we start with $S(A)$, then it is an idiom, thus by the above we have that $\mu^{(\_)}$ transforms any stable inflator into a pre-nucleus on $S(A)$. Note that $\mu^{(\_)}\colon I(A)\rightarrow P(I(A))$ is an embedding. Now consider any inflator $s\in I(A)$. Then, the pre-nucleus $\mu^{s}$ has a negation $\neg(\mu^{s})$ and this is a nucleus on $I(A)$.

\begin{prop}

\label{i10}
Let $s$ be any inflator on an idiom $A$. Then, $\neg(\mu^{s})\leq \mu^{t(s)}$ in $N(I(A))$.
\end{prop} 

\begin{proof}
We know that $\neg(\mu^{s})\wedge \mu^{s}=id_{I(A)}$, from item \ref{item7} of Section \ref{pre}, $\neg(\mu^{s})\mu^{t(s)}\wedge \mu^{s}\mu^{t(s)}=\mu^{t(s)}$. Now, take any inflator $d$ on $A$ we have, $(\mu^{s}\mu^{t(s)})(d)=\mu^{s}(dt(s))=(dt(s))s=d(t(s)s)=d\bar{d}=\bar{d}$ so $\mu^{s}\mu^{t(s)}=Tp$ is the top in $N(I(A))$. From this we obtain that $\neg(\mu^{s})\mu^{t(s)}=\mu^{t(s)}$, that is, $\neg(\mu^{s})\leq \mu^{t(s)}$.
\end{proof}

Compare the above with Lemma 20 of \cite{15}.
\begin{cor}
\label{i11}
If $s$ is any stable inflator on an idiom $A$, then $\neg s\leq t(s)$ in $C(A)$.  In particular, $\neg s\leq\jmath(s^{\infty}) $ in $N(A)$, and for any nucleus $j$ on $A$ we have $\neg j\leq \jmath(j)$ in $N(A)$.
\end{cor}

\begin{proof} 
From Proposition \ref{i10} we have that $\neg(\mu^{s})\leq \mu^{t(s)}$.  Evaluating in $d_{\underline{0}}$ we obtain that $\neg(\mu^{s})(d_{\underline{0}})\leq \mu^{t(s)}(d_{\underline{0}})$ and so $\neg(\mu^{s})(d_{\underline{0}})\leq t(s)$. 
Now, for the element 
\[\neg(\mu^{s})=\bigvee\{\varrho\in S(I(A))\mid \varrho \wedge \mu^{s}=id_{I(A)}\}=\bigvee\mathcal{P}\] 
we have that $\big(\bigvee\mathcal{P}\big)(d_{\underline{0}})=\bigvee\{\varrho(d_{\underline{0}})\mid \varrho\in \mathcal{P}\}$, that is, for any $\varrho\in\mathcal{P}$ we have that $\varrho(d_{\underline{0}})\wedge \mu^{s}(d_{\underline{0}})=d_{\underline{0}}$.  But since $\mu^{s}(d_{\underline{0}})=s$, then $\bigvee\mathcal{P}(d_{\underline{0}})\leq \neg s$, and since $\mu^{\neg s}\in \mathcal{P}$, then $\neg s=\bigvee\mathcal{P}(d_{\underline{0}})$, and therefore $\neg s\leq t(s)$. 
 The last affirmation follows from the fact that $\neg s^{\infty}=\neg s$, then $\neg s^{\infty}\leq t(s^{\infty})\leq \mathfrak{t}(s^{\infty})\leq \jmath(s^{\infty})$ and for a nucleus $\neg j\leq \mathfrak{t}(j)\leq \jmath{j}$.   
\end{proof}

\begin{prop}
\label{i12}
If $s$ is an inflator on an idiom $A$ and if $\neg(\mu^{s})= \mu^{t(s)}$, then $s$ is idempotent. In particular, if $s$ is a stable inflator, then it is a nucleus.
\end{prop}

\begin{proof}
Since $\neg(\mu^{s})\wedge \mu^{s}=\mu^{t(s)}\wedge \mu^{s}=id_{I(A)}$, evaluating in $d_{\underline{0}}$ we find that $t(s)\wedge s=d_{\underline{0}}$ and thus $t(s)s\wedge s^{2}=s$. But since $t(s)s=\bar{d}$, then $s^{2}=s$. The last statement is direct.
\end{proof}

\begin{prop}

\label{i13}
Let $A$ be an idiom. Then, for any stable inflator $s\in S(A)$ the following are equivalent:
\begin{itemize}
\item[{\rm (1)}]  $\neg s=\mathfrak{t}(s).$
\item[{\rm (2)}]  $s\leq \neg \mathfrak{t}(s).$
\end{itemize}
\end{prop}

\begin{proof}
If $\neg s=\mathfrak{t}(s)$, then $\neg \neg s=\neg \mathfrak{t}(s)$, and since $s\leq\neg\neg s$, then  $s\leq \neg \mathfrak{t}(s)$. 
Now, if $s\leq \neg \mathfrak{t}(s)$ implies that $s\wedge \mathfrak{t}(s)=d_{\underline{0}}$, then $\mathfrak{t}(s)\leq \neg s$, and the other comparison is just the Corollary \ref{i11}.
\end{proof}

\begin{cor}
\label{i14}
Let $A$ be an idiom. Then, the following are equivalent:
\begin{itemize}
\item[{\rm (1)}]  $\neg s=t(s)$ for all $s\in S(A)$.
\item[{\rm (2)}] $S(A)$ is a boolean algebra.
\end{itemize}
\end{cor}

\begin{proof}

$(1)\Rightarrow (2)$:  From Proposition \ref{i12} we obtain that $S(A)=P(A)=N(A)$.  Now, since the product of any two  pre-nuclei is a pre-nucleus then $ss'$ is idempotent, in particular $s\vee s'=ss'$. Hence, from $(1)$ we have that $(\neg s)s=\bar{d}=\neg s\vee s$, that is, $\neg s=t(s)$ is the complement of $s$, and thus $S(A)$ is boolean.

$(2)\Rightarrow (1)$: Every $s$ has a complement $\neg s$, and this implies that $(\neg s)s=\bar{d}$, that is, $t(s)\leq\neg s$. 
\end{proof}

\begin{dfn}
Let $d$ be an inflator on $A$. We say the idiom $A$ has $d$-\emph{length} if $d^{\infty}(\underline{0})=\bar{1}$.
\end{dfn}

\begin{obs}
Note also that if $d'$ is any other inflator with $d\leq d'$, then $d^{\infty}\leq d'^{\infty}$. Hence, if $d^{\infty}=\bar{d}$, then $d'^{\infty}=\bar{d}$, that is, $A$ has $d'$-length. In terms of totalizers this observation can be rephrased as follows: An idiom $A$ has $d$-length if $t(d^{\infty})=d_{\underline{0}}$. 
\end{obs}

There is other aspect of this theory: Consider any stable inflator $St$ on $N(A)$. Following Simmons (Definition 5.4 of \cite{28}) we said that a nucleus $j$ on $A$ has $St$-\emph{dimension} if $St^{\infty}(j)=\bar{d}$. If $d_{\underline{0}}$ has $St$-dimension, we say that $A$ has  $St$-dimension. Since $d^{A_{j}}_{\underline{0}}=j$, then a nucleus $j$ has $St$-dimension if and only if $A_{j}$ has $St$-dimension.

\begin{prop}
\label{i15}
Let $A$ be an idiom and $St$ any stable inflator on $N(A)$. Then, 
\[\widehat{\mu^{St^{\infty}(d_{\underline{0}})}}\leq St^{\infty}\] 
where $\widehat{\mu^{St^{\infty}(d_{\underline{0}})}}$ is the idempotent closure of the composition $((\_)^\infty\circ\mu^{St^{\infty}(d_{\underline{0}})})$.
\end{prop}
\begin{proof}
We have $St^{\infty}(d_{\underline{0}})\leq St^{\infty}(j)$ for all nuclei $j$ and $j\leq St^{\infty}(j)$. This last inequality implies that $jSt^{\infty}(j)=St^{\infty}(j)$ because $St^{\infty}(j)$ is in particular idempotent. Then, the first comparison implies that $jSt^{\infty}(d_{\underline{0}})\leq St^{\infty}(j)$, that is, $\mu^{St^{\infty}(d_{\underline{0}})}\leq St^{\infty}$, and hence $\widehat{\mu^{St^{\infty}(d_{\underline{0}})}}\leq St^{\infty}$.
\end{proof}

\begin{obs}
Proposition \ref{i15} implies that every nucleus $j$ on $A$, with $\widehat{\mu^{St^{\infty}(d_{\underline{0}})}}$-dimension has $St^{\infty}$-dimension. 

Now, recall that for every stable inflator $s$ on $A$ we can consider the pre-nucleus $\mu^{s^{\infty}}$  on $S(A)$. Note also that if $k$ is a pre-nucleus on $A$ such that $\mu^{s^{\infty}}(k)=\bar{d}$, then $\mathfrak{t}(s^{\infty})\leq k$. Then $k$ has $\mu^{s^{\infty}}$-dimension. From $\mu^{s^{\infty}}\leq (\_)^\infty\circ\mu^{s^{\infty}}\leq \widehat{\mu^{s^{\infty}}}$ we see that every $\mathfrak{t}(s^{\infty})\leq k$ has $\widehat{\mu^{s^{\infty}}}$-dimension.
\end{obs}
 

\begin{thm}
\label{i16}
Let $A$ be an idiom. The following affirmations are equivalent for a stable inflator $s\in S(A)$ and $J\in S(S(A))$:
\begin{itemize}
\item[{\rm (1)}] $\Xi(\mu^{s^{\infty}})(J^{\infty})=Tp$.
\item[{\rm (2)}]  $\mathfrak{t}(\mu^{s^{\infty}})\leq J^{\infty}$.
\item[{\rm (3)}]  $s^{\infty}$ has $J$-dimension. 
\end{itemize}
Here $\Xi(\mu^{s^{\infty}})(K)=K\mu^{s^{\infty}}$ and $Tp$ is the top in $S(S(A))$.
\end{thm}
\begin{proof}

$(1)\Leftrightarrow (2)$:  
$\Xi(\mu^{s^{\infty}})(J)=tp\Leftrightarrow J\mu^{s^{\infty}}=tp\Leftrightarrow \mathfrak{t}(\mu^{s^{\infty}})\leq J$.

$(2)\Rightarrow (3)$: $J(s^{\infty})=\bar{d}$ by $(2)$ and  evaluation on $d_{\underline{0}}$.  Thus $J^{\infty}(s^{\infty})=\bar{d}$.

$(3)\Rightarrow (2)$: 
By (3) we have $J^{\infty}(s^{\infty})=tp$. Then, $J^{\infty}((\mu^{s^{\infty}})(d_{\underline{0}}))=tp$ and thus $J^{\infty}\mu^{s^{\infty}}=tp$.
\end{proof}

\section{Some final remarks}\label{sec:sec4}

In this section we set down some observations about totalizers and questions related to equalizers.

Recall that for every idiom $A$ we have two fundamental inflators the socle, $soc=|\EuScript{S}mp|$ and the Cantor-Bendixson derivative $cbd=|\EuScript{C}mp|$ and this two are compared $soc\leq cbd$ thus $t(cbd)\leq t(soc)$. The following is not entirely new, the first two equivalences are proved in \cite{23} and the third is just a consequence of our analysis.
\begin{lem}
\label{r1}

Let $A$ be an idiom then the next statements are equivalent.
\begin{itemize}

\item[1] $A$ is strongly atomic, that is, every interval of $A$, $[a,b]$ contains a sub-interval, $a\leq c<d\leq b $ such that there is some $c<z\leq d$ with $[c,z]\in \EuScript{S}mp(\EuScript{O})$

\item[2] $soc^{\infty}(\underline{0})=\bar{1}$.

\item[3] $t(soc^{\infty})=d_{\underline{0}}$.
\end{itemize}
\end{lem}
The above result is also related to Corollary \ref{i911}.

Now from the preliminaries we have that for every idiom $A$, $Gab\leq Soc$ and $Boy\leq Cbd$, and all this inflators are below $Cbd$, here $Soc$ and $Cbd$ are Socle inflator and the Cantor-Bendixson derivative on $N(A)$ and for the construction of $Boy$ see for instance \cite{28}. The proof of the following is straightforward.

\begin{prop}\label{r2}
The following comparison holds for every idiom $A$.
\begin{itemize} 

\item[1] $t(Soc^{\infty})\leq t(Soc)\leq t(Gab)$. 

\item[2] $t(Cbd^{\infty})\leq t(Cbd)\leq t(Boy)$.  

\item[3] $t(Boy)\leq t(Gab)$.

\item[4] $t(Cbd)\leq t(Soc)$ and $t(Cbd^{\infty})\leq t(Soc^{\infty})$.
\end{itemize}
\end{prop}

\begin{obs}\label{obr1}
\begin{itemize}
\item[(1)] Suppose $t(Gab)=Id_{N(A)}$ then by $1$ of Proposition \ref{r2} we have $Gab=Soc=Soc^{\infty}=Tp$, that is, $N(A)$ have $Soc$-length, in particular $Gab(d_{\underline{0}})=\bar{d}$ and by Theorem 2.3 of \cite{28} we obtain that $Gab(d_{\underline{0}})=\bar{d}=soc^{\infty}$, that is to say $A$ has $soc$-length. In this case $3$ of the above proposition say $Gab=Boy$ which implies that $soc^{\infty}=cbd^{\infty}$.

\item[(2)] From $Gab\leq Soc$ we have $Gab(j)\leq Soc(j)$ for every nucleus, thus $t(Soc(j))\leq t(Gab(j))$ and $t(Gab^{\infty}(j))$. If $t(Gab(j))=d_{\underline{0}}$ then $Gab(j)=\bar{d}$ in particular $j$ has $Gab$-dimension and $soc^{\infty}_{j}=tp$, in this case $A_{j}$ is strongly atomic.
This situation also implies that $t(Soc(j))\leq t(soc^{\infty}_{j})\leq j_{*}t(soc_{A_{j}}^{\infty})j^{*}$ where the last comparison is for Proposition \ref{i9} so if $A_{j}$ is strongly atomic then $Soc(j)=tp$ in particular $t(Soc(d_{\underline{0}}))\leq t(soc^{\infty})\leq t(soc)$.
\end{itemize}
\end{obs}

It is important to mention that this kind of analysis can be applied to the (big) idiom $R$-pr of all preradicals over a module category $R$-$\Mod$, in particular this can be applied to the idiom $R$-$\lep$ of all left exact preradicals, note also that for every left exact preradical $\tau$ we can consider the next function $(\tau:\_):R\text{-}\lep\rightarrow R\text{-}\lep$ and from \cite{16} $(\tau:\_)$ is a pre-nucleus then $t(t((\tau:\_))(\tau))\leq t(\tau)\leq t((\tau:\_))(t(\tau))$ where $t(\tau)$ is the totalizer of $\tau$ in the sense of \cite{17}.

As the reader may notice, we do not have an analogue treatment like in the case of totalizers for equalizers thus further investigation is needed, but so far we can set down some properties of equalizers that give some lines for uses of these and insights to look at it.

\begin{dfn}\label{e1}
An inflator $d\neq \bar{d}$ on an idiom $A$ is $\wedge$\emph{-prime} if $k\wedge l=d$ then $k=d$ or $l=d$, for inflators $k$ and $l$.
\end{dfn}

From this we have that every $\wedge$-irreducible inflator is $\wedge$-prime.

\begin{lem}\label{e2}
If $d$ is an idempotent inflator and $\wedge$-prime then, $d$ is $\wedge$-irreducible.
\end{lem}

\begin{proof}
Let be $k_{1}, k_{2}\in D(A)$ such that $k_{1}\wedge k_{2}\leq d$, set $z_{1}=k_{1}\vee d$ y $z_{2}=k_{2}\vee d$ thus $d\leq z_{1}\wedge z_{2}=(k_{1}\vee d)\wedge (k_{2}\vee d)\leq (k_{1}d)\wedge (k_{2}d)=(k_{1}\wedge k_{2})d$. Therefore, since $d$ is idempotent $d=z_{1}\wedge z_{2}$. From the fact that $d$ is $\wedge$-prime we have the result.
\end{proof}

\begin{prop}\label{e3}
Suppose that $d$ is $\wedge$-prime then $e(d)$ is $\wedge$-prime.
\end{prop}
\begin{proof}
Note that being $d$ $\wedge$-prime then $e(d)\neq\bar{d}$, suppose $d_{1}\wedge d_{2}\leq e(d)$ then $(d_{1}\wedge d_{2})d=d$, thus $d_{1}d=d$ or $d_{2}d=d$, that is, $d_{1}\leq e(d)$ or $d_{2}\leq e(d)$.
\end{proof}

By Lemma \ref{i2} we know that for an inflator $d$ on an idiom $A$ its equalizer $e(d)$ is always idempotent, $e(d)\leq d$; and $e(d)=d$ if and only if $d$ is idempotent.

\begin{dfn}\label{e4}
Let $d$ be an inflator on an idiom $A$, the \emph{idempotent interval} of $d$ is \[[e(d),d^{\infty}]\]
where $d^{\infty}$ is the least idempotent above $d$.
\end{dfn}

This interval is trivial if and only if $d$ is idempotent.Recall that an element $a\neq{e}\in [a,b]$ in an interval, is \emph{large or essential} if $e\wedge y=a \Rightarrow y=a$ for all $a\leq y\leq b$. 

\begin{lem}\label{e5}
Let $d$ be a stable inflator on  $A$, then $d$ is an essential element of $[e(d),d^{\infty}]$.
This interval is considered in the set of stable inflators
\end{lem} 

\begin{proof}
Let $k$ be a stable inflator such that $e(d)\leq k\leq d^{\infty}$ and $d\wedge k=e(d)$ in this case the inflator $(\_)^{\infty}$ is stable then $e(d)=(d\wedge k)^{\infty}=d^{\infty}\wedge k^{\infty}=k^{\infty}$ thus $e(d)\leq k\leq k^{\infty}=e(d)$, that is, $e(d)=k$.
\end{proof}


\begin{thebibliography}{30}

\bibitem{1} J. Castro, J. R\'ios, M.L. Teply, Torsion theoretic dimensions and relative Gabriel correspondence. 
Journal of Pure and  Applied Algebra 
\em{178}, 101-114 (2003)

\bibitem{2} 
J. Castro, F. Raggi, J. R\'ios, J. Van den Berg, On atomic dimension in module categories. 
Communications in Algebra 
\em{33}(12), 4679-4692 (2005).

\bibitem{8} S. J. Golan,
      Torsion Theories. 
			Longman Scientific and Technical, New York, \em{29} (1986).


\bibitem{11} P. Johnstone, 
		Stone Spaces, 
		Cambridge studies in advanced mathematics, London (1992).

\bibitem{13} J.Picado, A. Pultr,
    Frames and Locales: Topology without points. 
		Birkh\"auser, Boston, (2010).

\bibitem{15} F. Raggi, J. R\'ios, R. Wisbauer,
The lattice structure of hereditary pretorsion classes.
Communications in Algebra
\em{29}(2), 541-556 (2001).

\bibitem{16} F. Raggi, J. R\'ios, H. Rinc\'on, R. Fern\'andez, C. Signoret,
The lattice structure of preradicals
Communications in Algebra 
\em{30}(3), 1533-1544 (2002).

\bibitem{17} F. Raggi, J. R\'ios, H. Rinc\'on, R. Fern\'andez, C. Signoret,
The lattice structure of preradicals II (Partitions)
Journal of Algebra and its Applications 
\em{1}(2), 201-214 (2002).

\bibitem{18} F. Raggi, J. R\'ios, H. Rinc\'on, R. Fern\'andez, C. Signoret,
The lattices structure of preradicals III: Operators
Journal of Pure and Applied Algebra 
\em{190}, 251-265 (2004).


\bibitem{23} H. Simmons, Near Discreteness of modules and spaces measured by Gabriel and Cantor. J. Pure and Applied Algebra. 
\em{56}, 119--162 (1989).

\bibitem{24} H. Simmons, A collection of notes on frames, 

\bibitem{25} H. Simmons, A decomposition theory for complete modular meet-continuous lattices. Algebra Universalis. 
\em{64}, 349--377, (2011).

\bibitem{26} H. Simmons, An introduction to idioms, 
\bibitem{27} H. Simmons, Cantor-Bendixson, socle and atomicity, 

\bibitem{28} H. Simmons, The Gabriel and the Boyle derivatives for a modular idiom, 


\bibitem{30} B. Stenstr\"om,
Rings of Quotients: An Introduction to Methods of Ring Theory, Springer Verlag, Berlin, (1975). 

\end{thebibliography}
\end{document}